\documentclass[11pt]{amsart}
\usepackage[dvipsnames,usenames]{color}
\usepackage{hyperref}
\usepackage{graphicx}
\usepackage{epsfig}
\usepackage[latin1]{inputenc}
\usepackage{amsmath}
\usepackage{amsfonts}
\usepackage{amssymb}
\usepackage{amsthm}
\usepackage{amscd}
\usepackage{verbatim}
\usepackage{subfigure}
\usepackage{pinlabel}
\usepackage{stmaryrd}
\usepackage{enumerate, enumitem}
\usepackage{todonotes}

\usepackage{tikz}
\usetikzlibrary{arrows}
\usetikzlibrary{decorations.pathreplacing}
\usepackage{verbatim}

\newtheorem{theorem}{Theorem}[section]
\newtheorem{Theorem}{Theorem}
\newtheorem{lemma}[theorem]{Lemma}
\newtheorem{proposition}[theorem]{Proposition}

\newtheorem{corollary}[theorem]{Corollary}
\newtheorem{Corollary}[Theorem]{Corollary}

\theoremstyle{definition}
\newtheorem{definition}[theorem]{Definition}

\theoremstyle{remark}

\def\F{\mathbb{F}}

\def\Z{\mathbb{Z}}

\def\cC{\mathcal{C}}

\def\CFKi{CFK^{\infty}}

\newcommand{\Pin}{\mathit{Pin}}

\def\im{\operatorname{im}}

\def \del {\partial}

\def\d{\partial}

\def\co{\colon}

\def\id{\textup{id}}

\title{}

\author[Kristen Hendricks]{Kristen Hendricks}
\thanks{The first author was partially supported by NSF grant DMS-1663778.}
\address{Department of Mathematics, Michigan State University, East Lansing, MI 48824}
\email{hendricks@math.msu.edu}
\numberwithin{equation}{section}

\subjclass[2013]{}
\author[Jennifer Hom]{Jennifer Hom}
\thanks{The second author was partially supported by NSF grant DMS-1552285 and a Sloan Research Fellowship.}
\address {School of Mathematics, Georgia Institute of Technology, Atlanta, GA 30332}
\email{hom@math.gatech.edu}

\title{A note on knot concordance and involutive knot Floer homology}

\begin{document}
\maketitle

\begin{abstract}
We prove that if two knots are concordant, then their involutive knot Floer complexes satisfy a certain type of stable equivalence.
\end{abstract}

\section{Introduction}

The knot Floer homology package of Ozsv\'ath-Szab\'o \cite{OSknots} and Rasmussen \cite{Rasmussenthesis} has many applications to concordance. For example, many different smooth concordance invariants can be extracted from the filtered chain homotopy type of the knot Floer complex, such as $\tau$ \cite{OS4ball}, $\Upsilon(t)$ \cite{OSS}, and $\nu^+$ \cite{HomWu}. Furthermore, the second author \cite{Homconcordance} showed that, modulo an appropriate equivalence relation, the set of knot Floer complexes forms a group, and that there is a homomorphism from the knot concordance group to this group. In \cite[Theorem 1]{Homsurvey}, she showed that if two knots are concordant, then their knot Floer complexes satisfy a certain type of stable equivalence.

Recently, Manolescu and the first author \cite{HendricksManolescu} used the conjugation symmetry on Heegaard Floer complexes to define involutive Heegaard Floer homology. They similarly considered the conjugation action on the knot Floer complex. Zemke \cite{Zemkesum} showed that, under an appropriate equivalence relation, the set of knot Floer complexes together with the extra structure given by the conjugation action form a group, and that there is a homomorphism from the knot concordance group to this group. The aim of this note is to prove an involutive analog of \cite[Theorem 1]{Homsurvey}. Throughout, $\F=\Z/2\Z$.

\begin{Theorem}\label{thm:splitting}
If $K$ is slice, then $(\CFKi(K), \iota_K)$ is filtered chain homotopic to 
\[ (\F[U, U^{-1}], \id) \oplus (A, \iota_A), \]
where $A$ is acyclic, i.e., $H_*(A)=0$.
\end{Theorem}

\begin{Corollary}\label{cor:splitting}
If $K_1$ and $K_2$ are concordant, then we have the following filtered chain homotopy equivalence
\[ (\CFKi(K_1), \iota_{K_1}) \oplus (A_1, \iota_{A_1}) \simeq (\CFKi(K_2), \iota_{K_2})\oplus (A_2, \iota_{A_2}), \] 
where $A_1, A_2$ are acyclic, i.e., $H_*(A_1)=H_*(A_2)=0$.
\end{Corollary}

%\subsection*{Organization}

\section*{Acknowledgements} We thank Tye Lidman, Charles Livingston, Ciprian Manolescu, Matt Stoffregen, and Ian Zemke for helpful conversations.

\section{Background}

In 2013, Manolescu introduced a $\Pin(2)$-equivariant version of Seiberg-Witten Floer homology and used it to resolve the Triangulation Conjecture~\cite{Manolescu16:triangulation}. Since then, several authors have given applications of this invariant, especially to the homology cobordism group~\cite{Manolescu14:intersection,Lin15:KO,Stoffregen:SFS,Stoffregen:sum,Stoffregen:remark}.
F.~Lin also gave a reformulation to monopole Floer homology, and deduced various applications~\cite{Lin:MB-Floer,Lin:Exact,Lin:higher-comp,Lin:correction,Lin:involutive-Kh}.

Two years later, Manolescu and the first author introduced a shadow of $\Pin(2)$-equivariant Seiberg-Witten Floer homology, called \emph{involutive Heegaard Floer homology}~\cite{HM:involutive}, in Ozsv\-\'ath-Szab\'o's Heegaard Floer homology~\cite{OS04:HolomorphicDisks}. Involutive Heegaard Floer homology has had a number of applications, again mainly to the homology cobordism group~\cite{HMZ:involutive-sum,BH:cuspidal,DM:involutive-plumbed, Zemkesum, HL:involutive-bordered}. 

Like ordinary Heegaard Floer homology, involutive Heegaard Floer homology has a version for knots: Manolescu and the first author associate to a knot $K$ an order-four symmetry $\iota_K$ on the knot Floer complex $\CFKi(K)$, and extract various concordance invariants from this data \cite{HM:involutive}. In \cite{Zemkesum}, Zemke studies the behavior of these complexes and the associated involutions under connected sum. In this section, we recap some of his definitions and results, in preparation for proving Theorem \ref{thm:splitting} in Section \ref{sec:involutive}. We begin with the following definition, which is a specialization of \cite[Definition 2.2]{Zemkesum}.

\begin{definition} We say that $(C, \del, B, \iota_C)$ is an \emph{$\iota_K$-complex} if
\begin{itemize}
\item $(C, \del)$ is a finitely-generated, free, $\Z$-graded, $(\Z\oplus \Z)$-filtered, $\F[U,U^{-1}]$-complex with a filtered basis $B$;
\item Given an element $x \in B$, $\del x = \sum_{y \in B} U^{n_y} y$ for some set of integers $n_y \geq 0$; 
\item The action of $U$ lowers homological grading by $2$ and each filtration level by $1$;
\item There is an isomorphism $H_*(C, \del) \cong \F[U,U^{-1}]$;
\item $\iota_C$ is a skew-filtered $U$-equivariant endomorphism of $C$;
\item $\iota_C^2 \simeq \mathrm{id} + \Phi_B \circ \Psi_B$, where $\Phi_B \co C \rightarrow C$ and $\Psi_B \co C \rightarrow C$ are formal derivatives of $\del$.

\end{itemize}

\end{definition} 

\noindent (For more on the definition of the maps $\Phi_B$ and $\Psi_B$, see \cite[p. 7]{Zemkesum}.) 

Typically we omit the differential and basis from the notation. This definition is not quite Zemke's; one can think of our $\iota_K$-complexes as the part of his $\iota_K$-complexes concentrated in Alexander grading zero \cite[Remark 2.3]{Zemkesum}. If $K$ is a knot, then $(\CFKi(K), \iota_K)$ can be made into an $\iota_K$-complex by picking a basis for $\CFKi(K)$. The following notion of equivalence between two $\iota_K$-complexes is particularly useful for studying concordance.

\begin{definition} \cite[Definition 2.4]{Zemkesum} \label{def:iotaK} Two $\iota_K$-complexes $(C_1, \iota_{C_1})$ and $(C_2, \iota_{C_2})$ are said to be \emph{locally equivalent} if there are filtered, grading-preserving $\F[U,U^{-1}]$-equivariant chain maps 
\begin{align*}
F \co C_1 \rightarrow C_2 \qquad \qquad G \co C_2 \rightarrow C_1
\end{align*}
such that
\begin{align*}
F \circ \iota_{C_1} \simeq \iota_{C_2} \circ F \qquad \qquad G \circ \iota_{C_2} \simeq \iota_{C_1} \circ G
\end{align*}
via skew-filtered $U$-equivariant chain homotopy equivalences. (If in addition $F \circ G \simeq \id$ and $G \circ F \simeq \id$ via filtered $U$-equivariant chain homotopy equivalences, the $\iota_K$-complexes are said to be homotopy equivalent.)
\end{definition}

One can define two possible products on the set of $\iota_K$-complexes, denoted $\times_1$ and $\times_2$, and given by
\begin{align*}
(C_1, \iota_{C_1})\times_1 (C_2, \iota_{C_2}) =(C_1 \otimes C_2, \iota_{C_1}\otimes \iota_{C_2} + (\Phi_{B_1}\otimes \Psi_{B_2})\circ (\iota_{C_1} \otimes \iota_{C_2})) \\
(C_1, \iota_{C_1})\times_2 (C_2, \iota_{C_2}) =(C_1 \otimes C_2, \iota_{C_1}\otimes \iota_{C_2} + (\Psi_{B_1} \otimes \Phi_{B_2})\circ (\iota_{C_1} \otimes \iota_{C_2})) \\
\end{align*}

Zemke shows that $(C_1, \iota_{C_1})\times_1 (C_2, \iota_{C_2})$ is filtered chain-homotopy equivalent to $(C_1, \iota_{C_1})\times_2 (C_2, \iota_{C_2})$. Following similar work in \cite{Stoffregen:SFS} and \cite{HMZ:involutive-sum}, Zemke further shows that either of these products makes the set of $\iota_K$-complexes up to the relationship of local equivalence into an abelian group $\mathfrak I_K$ \cite[Proposition 2.6]{Zemkesum}. One then obtains a homomorphism from $\cC$ the smooth knot concordance group to $\mathfrak I_K$ as follows.

\begin{proposition} \cite[Theorem 1.5]{Zemkesum}
Let $\cC$ be the smooth knot concordance group. The map 
\begin{align*}
\cC & \rightarrow \mathfrak I_K \\
K &\mapsto [\CFKi(K), \iota_K]
\end{align*}
is a well-defined group homomorphism.
\end{proposition}

\noindent Zemke \cite[Definition 2.5 and Proposition 2.6]{Zemkesum} shows that the inverse of $[C, \iota_C]$ is $[C^*, \iota^*_C]$, where $C^*=\operatorname{Hom}_{\F[U, U^{-1}]}(C, \F[U, U^{-1}])$, the map $\iota^*$ is the dual of $\iota$, and $B^*$ is a dual basis to $B$. The identity element of $\mathfrak I_K$ is $[\F[U,U^{-1}], \mathrm{id}]$.

\section{Proof of Theorem} \label{sec:involutive}
Since Zemke \cite[Theorem 1.5]{Zemkesum} showed that concordant knots have locally equivalent $\iota_K$-complexes,
Theorem \ref{thm:splitting} and Corollary \ref{cor:splitting} follow immediately from the following proposition and corollary.

\begin{proposition}\label{prop:split}
If $(C, \iota_C)$ is locally equivalent $(\F[U, U^{-1}], \id)$, then $(C, \iota_C)$ is filtered chain homotopy equivalent to
\[ (\F[U, U^{-1}], \id) \oplus (A, \iota_A), \]
where $A$ is some acyclic complex, i.e., $H_*(A)=0$.
\end{proposition}

\begin{corollary}\label{cor:stableequiv}
If $(C_1, \iota_{C_1})$ is locally equivalent to $(C_2, \iota_{C_2})$, then we have the following filtered chain homotopy equivalence
\[ (C_1, \iota_{C_1}) \oplus (A_1, \iota_{A_1}) \simeq (C_2, \iota_{C_2}) \oplus (A_2, \iota_{A_2}), \]
for some acyclic complexes $A_1$ and $A_2$.
\end{corollary}

\begin{proof}[Proof of Corollary \ref{cor:stableequiv}]
If $(C_1, \iota_{C_1})$ and $(C_2, \iota_{C_2})$ are locally equivalent, then by \cite[Proposition 2.6]{Zemkesum} $(C_1, \iota_{C_1}) \times (C^*_2, \iota^*_{C_2})$ is locally equivalent to $(\F[U, U^{-1}], \id)$, where $\times$ denotes either $\times_1$ or $\times_2$. Then by Proposition \ref{prop:split}, $(C_1, \iota_{C_1}) \times (C^*_2, \iota^*_{C_2})$ is filtered chain homotopy equivalent to $(\F[U, U^{-1}], \id) \oplus (A, \iota_A)$. 

Consider $(C_1, \iota_{C_1}) \times (C^*_2, \iota^*_{C_2}) \times (C_2, \iota_{C_2})$. By \cite[Theorem 1.1]{Zemkesum}, the product $\times$ respects splittings and $(\F[U, U^{-1}], \id)$ is the identity element with respect to $\times$. Then
\begin{align*}
	((C_1, \iota_{C_1}) \times (C^*_2, \iota^*_{C_2})) \times (C_2, \iota_{C_2}) &\simeq \big((\F[U, U^{-1}], \id) \oplus (A, \iota_A)\big) \times  (C_2, \iota_{C_2})\\
		&\simeq (C_2, \iota_{C_2}) \oplus (A', \iota_A'),
\end{align*}
where $(A', \iota_A') = (A, \iota_A) \times (C_2, \iota_{C_2})$. Similarly, for some acyclic complex $D$, we have
\begin{align*}
	(C_1, \iota_{C_1}) \times ((C^*_2, \iota^*_{C_2}) \times (C_2, \iota_{C_2})) &\simeq (C_1, \iota_{C_1}) \times \big((\F[U, U^{-1}], \id) \oplus (D, \iota_D)\big)\\
		&\simeq (C_1, \iota_{C_1}) \oplus (D', \iota_D'),
\end{align*}
where $(D', \iota_D') = (D, \iota_D) \times (C_2, \iota_{C_2})$. This concludes the proof of the corollary.
\end{proof}

Using the language of local equivalence, we reprove \cite[Theorem 1]{Homsurvey}.

\begin{lemma}\label{lem:split}
If $(C, \iota_C)$ is locally equivalent $(\F[U, U^{-1}], \id)$, then $C$ is filtered chain homotopic to $\F[U, U^{-1}] \oplus A$.
\end{lemma}

\begin{proof}[Proof of Lemma \ref{lem:split}]
Since $(C, \iota_C)$ and $(\F[U, U^{-1}], \id)$ are locally equivalent, there exist grading-preserving, filtered chain maps
\begin{align*}
F \co &\F[U, U^{-1}] \rightarrow C \\
G \co &C \rightarrow \F[U, U^{-1}]
\end{align*}
that induce isomorphisms on homology. Since $\F[U, U^{-1}]$ is isomorphic to its homology, $G$ is surjective and $G \circ F = \id$. Then a standard algebra argument shows that $C$ is filtered isomorphic to $\F[U, U^{-1}] \oplus \ker G$. Namely, $\Phi \co \F[U, U^{-1}] \oplus \ker G \rightarrow C$ given by $(x, y) \mapsto x+y$ and $\Psi \co C \rightarrow \F[U, U^{-1}] \oplus \ker G$ given by $z \mapsto ( F \circ G(z), z+F\circ G(z))$ provide the necessary isomorphisms, where we identify $\F[U, U^{-1}]$ with $\im F$.
\end{proof}

Notice that in general $\iota_C$ does not respect the splitting in the above lemma. However, we will show that $\iota_C$ is homotopic to a map that does split.

\begin{proof}[Proof of Proposition \ref{prop:split}]
By Lemma \ref{lem:split}, we may assume $C$ is of the form $\F[U, U^{-1}] \oplus A$. Since $(C, \iota_C)$ and $(\F[U, U^{-1}], \id)$ are locally equivalent, there exist grading-preserving, filtered chain maps
\begin{align*}
F \co &(\F[U, U^{-1}], \id) \rightarrow (C, \iota_C) \\
G \co &(C, \iota_C) \rightarrow (\F[U, U^{-1}], \id)
\end{align*}
such that $F \circ \id \simeq \iota_C \circ F$ via a skew-filtered chain homotopy $H_F$ and $G \circ \iota_C \simeq \id \circ G$ via a skew-filtered chain homotopy $H_G$. 

We consider the splitting given in Lemma \ref{lem:split}. Let $p_i \co \F[U, U^{-1}] \oplus A \rightarrow \F[U, U^{-1}] \oplus A$ denote projection onto the $i^{\textup{th}}$ factor. We have that $p_1=F \circ G$ and $p_2= \id + F \circ G$. 

Define
\[ \iota'_C (x, y) = (x, 0) + p_2 \circ \iota_C(0, y). \]
We claim that $\iota_C \simeq \iota'_C$ via the homotopy $J=H_F \circ G + F \circ H_G \circ p_2$. Indeed, 
\begin{align*}
	\iota_C(x, y) + \iota'_C(x,y) &= \iota_C(x, 0) + \iota_C(0,y) + (x,0) + p_2 \circ \iota_C(0, y) \\
		&= \iota_C(x, 0) + \iota_C(0,y) + (x,0) + \iota_C(0, y) + F \circ G \circ \iota_C (0, y) \\
		&= \iota_C \circ F \circ G (x, y) + F \circ \id \circ G (x,y) + F \circ G \circ \iota_C (0, y) + F \circ \id \circ G(0,y)\\
		&= \d \circ H_F \circ G (x, y) + H_F \circ \d \circ G  (x,y) +  F \circ \d \circ H_G (0, y) + F \circ H_G \circ \d (0, y) \\
		&= \d \circ J(x,y) + J \circ \d (x,y).
\end{align*}
It is straightforward to check that $\iota'_C$ respects the splitting $\F[U, U^{-1}] \oplus A$ and that it is the identity on the first factor, as desired.
\end{proof}

\bibliographystyle{amsalpha}
\bibliography{bib}

\providecommand{\bysame}{\leavevmode\hbox to3em{\hrulefill}\thinspace}
\providecommand{\MR}{\relax\ifhmode\unskip\space\fi MR }
\providecommand{\MRhref}[2]{%
  \href{http://www.ams.org/mathscinet-getitem?mr=#1}{#2}
}
\providecommand{\href}[2]{#2}
\begin{thebibliography}{HMZ17}

\bibitem[BH16]{BH:cuspidal}
Maciej Borodzik and Jennifer Hom, \emph{Involutive {H}eegaard {F}loer homology
  and rational cuspidal curves}, 2016, preprint, arXiv:1609.08303.

\bibitem[DM17]{DM:involutive-plumbed}
Irving Dai and Ciprian Manolescu, \emph{Involutive {H}eegaard {F}loer homology
  and plumbed three-manifolds}, 2017, preprint, arXiv:1704.02020.

\bibitem[HL17]{HL:involutive-bordered}
Kristen Hendricks and Robert Lipshitz, \emph{Involutive bordered {F}loer
  homology}, 2017, preprint, arXiv:1706.06557.

\bibitem[HM17a]{HendricksManolescu}
Kristen Hendricks and Ciprian Manolescu, \emph{Involutive {H}eegaard {F}loer
  homology}, Duke Math. J. \textbf{166} (2017), no.~7, 1211--1299.

\bibitem[HM17b]{HM:involutive}
\bysame, \emph{Involutive {H}eegaard {F}loer homology}, Duke Math. J.
  \textbf{166} (2017), no.~7, 1211--1299.

\bibitem[HMZ17]{HMZ:involutive-sum}
Kristen Hendricks, Ciprian Manolescu, and Ian Zemke, \emph{A connected sum
  formula for involutive {H}eegaard {F}loer homology}, Selecta Math. (2017),
  1--63.

\bibitem[Hom14]{Homconcordance}
Jennifer Hom, \emph{The knot {F}loer complex and the smooth concordance group},
  Comment. Math. Helv. \textbf{89} (2014), no.~3, 537--570.

\bibitem[Hom17]{Homsurvey}
\bysame, \emph{A survey on {H}eegaard {F}loer homology and concordance}, J.
  Knot Theory Ramifications \textbf{26} (2017), no.~2, 1740015, 24.

\bibitem[HW16]{HomWu}
Jennifer Hom and Zhongtao Wu, \emph{Four-ball genus bounds and a refinement of
  the {O}zv\'ath-{S}zab\'o tau invariant}, J. Symplectic Geom. \textbf{14}
  (2016), no.~1, 305--323.

\bibitem[Lin14]{Lin:MB-Floer}
Francesco Lin, \emph{A {M}orse-{B}ott approach to monopole {F}loer homology and
  the {T}riangulation conjecture}, 2014, preprint, arXiv:1404.4561.

\bibitem[Lin15a]{Lin:Exact}
\bysame, \emph{The surgery exact triangle in {P}in(2)-monopole {F}loer
  homology}, 2015, preprint, arXiv:1504.01993.

\bibitem[Lin15b]{Lin15:KO}
Jianfeng Lin, \emph{Pin(2)-equivariant {KO}-theory and intersection forms of
  spin 4-manifolds}, Algebr. Geom. Topol. \textbf{15} (2015), no.~2, 863--902.

\bibitem[Lin16a]{Lin:involutive-Kh}
Francesco Lin, \emph{{K}hovanov homology in characteristic two and involutive
  monopole {F}loer homology}, 2016, preprint, arXiv:1610.08866.

\bibitem[Lin16b]{Lin:correction}
\bysame, \emph{{M}anolescu correction terms and knots in the three-sphere},
  2016, preprint, arXiv:1607.05220.

\bibitem[Lin16c]{Lin:higher-comp}
\bysame, \emph{{$\mathrm{Pin}(2)$}-monopole {F}loer homology, higher
  compositions and connected sums}, 2016, preprint, arXiv:1605.03137.

\bibitem[Man14]{Manolescu14:intersection}
Ciprian Manolescu, \emph{On the intersection forms of spin four-manifolds with
  boundary}, Math. Ann. \textbf{359} (2014), no.~3-4, 695--728.

\bibitem[Man16]{Manolescu16:triangulation}
\bysame, \emph{Pin(2)-equivariant {S}eiberg-{W}itten {F}loer homology and the
  triangulation conjecture}, J. Amer. Math. Soc. \textbf{29} (2016), no.~1,
  147--176.

\bibitem[OS03]{OS4ball}
Peter Ozsv{\'a}th and Zolt{\'a}n Szab{\'o}, \emph{Knot {F}loer homology and the
  four-ball genus}, Geom. Topol. \textbf{7} (2003), 615--639.

\bibitem[OS04]{OSknots}
\bysame, \emph{Holomorphic disks and knot invariants}, Adv. Math. \textbf{186}
  (2004), no.~1, 58--116.

\bibitem[OSS17]{OSS}
Peter~S. Ozsv\'ath, Andr\'as~I. Stipsicz, and Zolt\'an Szab\'o,
  \emph{Concordance homomorphisms from knot {F}loer homology}, Adv. Math.
  \textbf{315} (2017), 366--426.

\bibitem[O{\relax Sz}04]{OS04:HolomorphicDisks}
Peter~S. Ozsv{\'a}th and Zolt{\'a}n {\relax Sz}ab{\'o}, \emph{Holomorphic disks
  and topological invariants for closed three-manifolds}, Ann. of Math. (2)
  \textbf{159} (2004), no.~3, 1027--1158.

\bibitem[Ras03]{Rasmussenthesis}
Jacob~Andrew Rasmussen, \emph{Floer homology and knot complements}, ProQuest
  LLC, Ann Arbor, MI, 2003, Thesis (Ph.D.)--Harvard University.

\bibitem[Sto15a]{Stoffregen:sum}
Matthew Stoffregen, \emph{{M}anolescu invariants of connected sums}, 2015,
  preprint, arXiv:1510.01286.

\bibitem[Sto15b]{Stoffregen:SFS}
\bysame, \emph{Pin(2)-equivariant {S}eiberg-{W}itten {F}loer homology of
  {S}eifert fibrations}, 2015, preprint, arXiv:1505.03234.

\bibitem[Sto16]{Stoffregen:remark}
\bysame, \emph{A remark on {$\mathrm{Pin}(2)$}-equivariant {F}loer homology},
  2016, preprint, arXiv:1605.00331.

\bibitem[Zem17]{Zemkesum}
Ian Zemke, \emph{Connected sums and involutive link {F}loer homology}, 2017,
  preprint, arXiv:1705.01117.

\end{thebibliography}

\end{document}